\documentclass[a4paper,12pt]{amsart}
\usepackage{amsfonts,amsmath,amsthm, amssymb}
\usepackage{latexsym, euscript, epic, eepic}
\usepackage{graphicx,amssymb}
\usepackage{enumerate}
\usepackage{amsmath}
\numberwithin{equation}{section}
\newtheorem{prop}{Proposition}[section]
\newtheorem{thm}{Theorem}[section]
\newtheorem{lem}{Lemma}[section]
\newtheorem{rem}{\bf Remark}[section]

\begin{document}
	\title{Non-dense orbit sets carry full metric mean dimension}
	
	\author{Jiao Yang$^1$, Ercai Chen$^1$, Xiaoyao Zhou*$^1$ }  
	\address
	{1.School of Mathematical Sciences and Institute of Mathematics, Nanjing Normal University, Nanjing 210023, Jiangsu, P.R.China}
	\email{jiaoyang6667@126.com}
	\email{ecchen@njnu.edu.cn}
	\email{zhouxiaoyaodeyouxian@126.com}
	\date{}	
	\renewcommand{\thefootnote}{}
\footnotetext{*corresponding author}
\date{}
\maketitle

\begin{center}
	\begin{minipage}{120mm}
		\small {\bf Abstract.} Let $(X,d)$ be a compact metric space, $f:X\rightarrow X$ be a continuous transformation with
		the specification property. we consider non-dense orbit set $E(z_0)$ and show that for any non-transitive point $z_0\in X$, this set $E(z_0)$ is empty or carries full Bowen upper and lower metric mean dimension.
	\end{minipage}
\end{center}

\section{Introduction}
A number $\lambda$ is called badly approximable if $|\lambda-\frac{p}{q}|>\frac{c}{q^2}$ for some $c>0$ and all rational numbers $\frac{p}{q}$. In 1931, Jarn\'{\i}k \cite{J31} proved that the set of badly approximable  numbers is of full Hausdorff dimension. In 1997,  Abercrombie and Nair \cite{AN97} proved the non-dense set for the expanding rational map of the Riemann sphere acting on its Julia set $J$ has full Hausdorff dimension. In fact, authors \cite{KTV06} generalized the result of \cite{AN97} by some more general systems. 

Let $(X,d,f)$ be a topological dynamical system, where $(X,d)$ is a compact metric space and $f:X\rightarrow X$ is a continuous map. For any $x\in X$, let $O_f(x)$ denote the orbit of $x$, i.e., $O_f(x):=\{x,f(x),\cdots,f^n(x),\cdots\}$. For any $z_0\in X$, we define
$$E(z_0)=\{x\in X:z_0\notin \overline{\{f^n(x):n\geq 0\}}\},$$
where $x\in E(z_0)$ indicates $z_0$ is badly approximated by the orbit of $x$. When $f$ is Guass map, $E(0)$ is just the set of the badly approximable numbers. By the definition, any point in $E(z_0)$ has a non-dense forward orbit in $X$. Recently, Zhao, Yang and Zhou \cite{ZYZ} showed that $E(z_0)$ can have full topological pressure.

It should be noted that on a compact smooth manifold with dimension greater than one, homeomorphisms with infinite topological entropy are $C^0$ generic \cite{Y80}. Recently, Bobok and Troubetzkoy \cite{BT20} showed that in the space of continuous non-invertible maps of the unit interval preserving the Lebesgue measure, which is equipped with the uniform metric, the functions satisfying specification property and infinite topological entropy form a dense $G_\delta$ set. Thus, a more subtle question arises naturally: Given a system having both the specification property and infinite topological entropy, does $E(z_0)$ carry more information besides infinite Bowen topological entropy?


Mean topological dimension introduced by Gromov \cite{G99} is a new topological invariant in topological dynamical systems.  
Later, Lindenstrauss and Weiss \cite{LW00} introduced the metric mean dimension, which is a metric version of the mean dimension. Metric mean dimension, similarly to the topological entropy, measures the complexity of systems with infinite entropy. Mean dimension has applications to topological dynamics \cite{LW00,L99}. Similar as the topological entropy, the metric mean dimension has a strong connection
with ergodic theory, and lots of variational principles have been established, see \cite{GS21,LT19}.                                                                                                                                                                                                                                                                                                                                                                                                                                                                                                                                                                                                                                                                                                                                                                                                                                                                                                                                                                                                                                                                                                                                                                                                                                                                                                                                                                                                                                                                                                                                                                                                                                                                                                                                                                                                                                                                                                                                                   

Let us go to back to our question. Recently, we state our result precisely

\subparagraph*{Main result.}In this paper, we consider a dynamical system $(X,d,f)$ satisfy specification, i.e., for
any $\epsilon>0,$ there exists an integer
$m=m(\epsilon)$ such that for arbitrary finite intervals $\{I_{j}=[a_{j},b_{j}]\}_{j=i}^k$ with  $a_{j+1}-b_j\geq m$ for $j=1,2,\cdots,k-1$ and any $x_{1},\cdots,x_{k}\in X$, there exists a point $x\in X$  such that
\begin{align*}
	d(f^{p+a_{j}}(x),f^{p}(x_{j}))<\epsilon\;\;\text{for all}\;\;p=0,1,\cdots,b_{j}-a_{j}\;\;\text{and every}\;\;j=1,2,\cdots,k.
\end{align*}
For convenience, we call $\{[b_{j},a_{j+1}],j=1,\cdots,k-1\}$ the gaps. 

\begin{thm}\label{th1}
	Suppose that $(X,d,f)$ be a dynamical system
	with specification property. For any non-transitive point $z_0\in X$, i.e. $\overline{O_f(z_0)}\neq X$, then either $E(z_0)=\emptyset$ or
	\begin{equation}\label{gs1}
		\begin{split}
			{\rm\overline{mdim}}_M^B(E(z_0),f,d)={\rm\overline{mdim}}_M(X,f,d)\\
			{\rm\underline{mdim}}_M^B(E(z_0),f,d)={\rm\underline{mdim}}_M(X,f,d).	
		\end{split}
	\end{equation} 
\end{thm}


\section{Basic notions and definitions}


Let $n\in \mathbb{N}$. For $x,y$, we define the $n$th Bowen metric $d_n$ on $X$ as
$$d_{n}(x,y)=\max\{d(f^{i}(x),f^{i}(y)):i=0,\cdots,n-1\}.$$ 
For each $\epsilon>0$, the Bowen ball of radius $\epsilon$ and order $n$ in the metric $d_n$ around $x$ is given by 
\begin{equation*}
	B_{n}(x,\epsilon)=\{y\in X:d_{n}(x,y)<\epsilon\}.
\end{equation*}

Now given $Z\subseteq X,\epsilon>0$ and $N\in \mathbb{N}$. For each $\lambda\in\mathbb{R},$ let
\begin{equation*}\begin{split}
		m(Z,\lambda,N,\epsilon)&=\inf\limits_{\varGamma}\left\lbrace \sum\limits_{i\in I}\exp\left(-\lambda n_i\right)\right\rbrace ,\\
\end{split}\end{equation*}
where the infimum is taken over all finite or countable collection  $\varGamma=\{B_{n_i}(x_i,\epsilon)\}_{i\in I}$ such that $Z\subseteq \cup_{i\in I}B_{n_i}(x_i,\epsilon)$ and $\min\{n_i: i \in I\}\geq N$. Note that $m(Z,\lambda,N,\epsilon)$ does not decrease as N increases, and therefore the following limit exists
\begin{align*}
	m(Z,\lambda,\epsilon)&=\lim\limits_{N\to\infty}m(Z,\lambda,\varphi,N,
	\epsilon).
\end{align*}
The function is non-increasing in $\lambda$ and takes value $\infty$ and $0$ at all but at most one value of $\lambda$. Denoting the critical value of $\lambda$ by
\begin{align*}
	h^B_{top}(Z,f,\epsilon)&=\inf\{\lambda\in\mathbb{R}:m(Z,\lambda,\epsilon)=0\}\\
	&=\sup\{\lambda\in\mathbb{R}:m(Z,\lambda,\epsilon)=\infty\}.
\end{align*}
This implies that $m(Z,\lambda,\epsilon)=\infty,$ when
$\lambda<h^B_{top}(Z,f,\epsilon),$ and $m(Z,\lambda,\epsilon)=0$ when $s>h^B_{top}(Z,f,\epsilon)$. Note that $m(Z,h^B_{top}(Z,f,\epsilon),\epsilon)$ could be $\infty,0$ or some positive finite number. The Bowen topological
entropy is defined by $h^B_{top}(Z,f)=\lim\limits_{\epsilon\to 0}h^B_{top}(Z,f,\epsilon)$ (see \cite{TV03}). The Bowen upper and lower metric mean dimension of $f$ on $Z$ with respect to $d$ are respectively defined by
\begin{align*}
	{\rm\overline{mdim}}^B_M(Z,f,d)&=\limsup \limits_{\epsilon\to 0}\frac{h^B_{top}(Z,f,\epsilon)}{\lvert\log\epsilon\rvert},\\
	{\rm\underline{mdim}}^B_M(Z,f,d)&=\liminf \limits_{\epsilon\to 0}\frac{h^B_{top}(Z,f,\epsilon)}{\lvert\log\epsilon\rvert}.
\end{align*}

The classical metric mean dimension is defined as follows. Given $n\in \mathbb{N}$ and $\epsilon>0$. A set $E\subset X$ is called an
$(n,\epsilon)$ separated set for $X$ if for every $x\neq y\in
E$, we have $d_{n}(x,y)>\epsilon$. Define $s(f,X,n,\epsilon)$ to be the largest cardinality of an $(n,\epsilon)$ separated set of $X$. Notice that $s_{sep}(f,X,n,\epsilon)$ is finite by compactness.The upper and lower metric mean dimension of $f$ with respect to $d$ respectively are given by
\begin{align*}
	{\rm\overline{mdim}}_M(X,f,d)&=\limsup\limits_{\epsilon\to 0}\frac{\limsup\limits_{n\to\infty}\frac{1}{n}\log s(f,X,n,\epsilon)}{|\log\epsilon|},\\
	{\rm\underline{mdim}}_M(X,f,d)&=\liminf\limits_{\epsilon\to 0}\frac{\limsup\limits_{n\to\infty}\frac{1}{n}\log s(f,X,n,\epsilon)}{|\log\epsilon|}.
\end{align*}
It is clear that the metric mean dimension vanishes if topological entropy is finite.


\begin{rem}
	If $Z_{1}\subset Z_{2}\subset X$, then 	
	\begin{align}\label{gs2.1}
		{\rm\overline{mdim}}^B_M(Z_1,f,d)\leq
		{\rm\overline{mdim}}^B_M(Z_2,f,d),\;\;
		{\rm\underline{mdim}}^B_M(Z_1,f,d)\leq {\rm\underline{mdim}}^B_M(Z_2,f,d).
	\end{align}
\end{rem}

The following proposition complete proof is shown in \cite{LL22} appendix.
\begin{prop}\label{prop}
	For any $f$-invariant and compact nonempty subset $Z\subset X$, one has
	\begin{align*}
		{\rm\overline{mdim}}^B_M(Z,f,d)={\rm\overline{mdim}}_M(Z,f,d),\;\;\;{\rm\underline{mdim}}^B_M(Z,f,d)={\rm\underline{mdim}}_M(Z,f,d).
	\end{align*}
\end{prop}

\section{Proof of Theorem \ref{th1}}
In this section, let's turn to prove our main result.
\subsection{Proof for the case of the upper metric mean dimension}
we assume $E(z_0)=\emptyset$ and show (\ref{gs1}). We first consider the case of the upper metric mean dimension. Note that $E(z_0)\subset X$, and therefore  
$${\rm\overline{mdim}}^B_M(E(z_0),f,d)\leq{\rm\overline{mdim}}^B_M(X,f,d).$$
Proposition \ref{prop} implies that
\begin{align*}
	{\rm\overline{mdim}}^B_M(E(z_0),f,d)\leq{\rm\overline{mdim}}_M(X,f,d).
\end{align*}
For any constant $C<{\rm\overline{mdim}}_M(X,f,d)$,
we only need to show that
\begin{align}\label{gs3.1}
	{\rm\overline{mdim}}^B_M(E(z_0),f,d)\geq C.
\end{align}  

Firstly, since $z_0$ is non-transitive point, we can choose $y\in X$ and $\epsilon_0>0$ such that
\begin{align}\label{gs3.2}
	d(y,\overline{O_f(z_0)})\geq 2\epsilon_0.
\end{align}  
Fix $\gamma>0$ we can choose a $\epsilon<\epsilon_0$ and a sequence $\{n_k\}_{k\geq1}\subset\mathbb{N}$ such that there exists maximal $(n_k,7\epsilon)$-separated set $\mathcal{S}_k$ of $X$ which is always a  $(n_k,7\epsilon)$-spanning set with 
\begin{align}\label{gs3.3}
	\#\mathcal{S}_k\geq\exp\left(n_k(C-\gamma)\right)|\log7\epsilon|,
\end{align}
\begin{align}\label{gs3.4}
	\frac{h^B_{top}(E(z_0),f,\epsilon)}{|\log\epsilon|}\leq{\rm\overline{mdim}}^B_M(E(z_0),f,d)+\gamma
\end{align}
and
\begin{align}\label{gs3.5}
	({\rm\overline{mdim}}^B_M(E(z_0),f,d)+\gamma)\cdot\frac{|\log\epsilon|}{|\log7\epsilon|}\leq{\rm\overline{mdim}}^B_M(E(z_0),f,d)+2\gamma
\end{align}
by the definitions  $${\rm\overline{mdim}}_M(X,f,d)=\limsup\limits_{\epsilon\to 0}\frac{\limsup\limits_{n\to\infty}\frac{1}{n}\log s(f,X,n,\epsilon)}{|\log\epsilon|},$$
and 
\begin{align*}
	{\rm\overline{mdim}}^B_M(Z,f,d)&=\limsup \limits_{\epsilon\to 0}\frac{h^B_{top}(Z,f,\epsilon)}{\lvert\log\epsilon\rvert}.
\end{align*}

\subsubsection{\bf Construction of the Moran-like fractal $\mathcal{F}$.}

Choose $M>0$ such that 
\begin{align}\label{gs3.6}
	\frac{2m(\epsilon)C}{M+2m(\epsilon)}<\gamma.
\end{align}
Without lose of generality, we can assume that $M<n_1$. Let $c_k=\lceil\frac{n_k}{M}\rceil$ then we break the $n_k$ orbit of $x\in \mathcal{S}_k$ as follows
\begin{align*}
	\{x,f(x),\cdots,f^{M-1}(x)\}\cup\{f^{M}(x),\cdots,f^{2M-1}(x)\}\cup\cdots\\ \cup\{f^{(c_k-1)M}(x),f^{(c_k-1)M+1}(x),\cdots,f^{n_k}(x)\}.
\end{align*}
Now we define the point pair $(x,n)\in X\times\mathbb{N}$ by $(x,n):=\{x,f(x),\cdots,f^{n-1}(x)\}$.

We insert $y$ into every gap, in fact, we translate the point pair $(x,n_k)$ to 
\begin{align*}
	(x,M),(y,1),(f^{M}(x),M),(y,1),\cdots,(y,1),(f^{(c_k-1)M}(x),n_k-(c_k-1)M).
\end{align*}
Denote $m:=m(\epsilon)$. By the specification property, there exists $y'\in X$ such that 
\begin{align*}
	d_M(y',x)<\epsilon,d(f^{M+m}y',y)<\epsilon,\cdots,d_M(f^{(j-1)(M+2m+1)}y',f^{(j-1)M}x)<\epsilon\\
	d(f^{(j-1)(M+2m+1)+M+m}y',y)<\epsilon,\cdots,d_{n_k-(c_k-1)M}(f^{(c_k-1)(M+2m+1)+M+m}y',f^{(c_k-1)M}x)<\epsilon
\end{align*}
i.e., the following set is a non-empty,
\begin{align*}
	B(x,n_k,\epsilon;y)=&\bigcap\limits_{j=1}^{c_k-1}\left\lbrace f^{-(j-1)(M+2m+1)}B_M(f^{(j-1)M}x,\epsilon)\cap f^{-(j-1)(M+2m+1)-M-m}B(y,\epsilon) \right\rbrace\\
	&\cap
	f^{-(c_k-1)(M+2m+1)-M-m}B_{n_k-(c_k-1)M} (f^{(c_k-1)M}x)\neq\emptyset.
\end{align*}
From the above setting, we define $\hat{n}_k:=n_k+(c_k-1)(2m+1)$ which denotes the length of the orbits in the set $B(x,n_k,\epsilon;y)$.

Next, we can choose a sequence and $N_k$ increasing to $\infty$ with $N_0=0$. We enumerate the points in the sets $\mathcal{S}_k$ provided by (\ref{gs3.3}) and write them as follows
\begin{align*}
	\mathcal{S}_k=\{x_i:i=1,2,\cdots,\#\mathcal{S}_k\}.
\end{align*}
We enumerate the points in the set $\mathcal{S}_k$ and consider the set $\mathcal{S}_k$. Let $\overline{x}_k=(x_1^k,\cdots,x_{N_k}^k)\in \mathcal{S}_k^{N_k}$ where $\mathcal{S}_k^{N_k}=\mathcal{S}_k\times\cdots\times\mathcal{S}_k$. Now we set $t_1:=\hat{n}_1N_1+(N_1-1)m$ and if $t_k$ have been defined, we define $t_{k+1}:=t_k+N_{k+1}(\hat{n}_{k+1}+m)$. By the specification property, we have  
\begin{align*}
	B(\overline{x}_1,\cdots,\overline{x}_k;y)=\bigcap\limits_{i=1}^{k}\bigcap\limits_{j=1}^{N_i}f^{-t_{i-1}-m-(j-1)(\hat{n}_i+m)}B(x_j^i,n_i,\epsilon;y)\neq\emptyset.
\end{align*}
We define $\mathcal{F}_k$ by 
\begin{align*}
	\mathcal{F}_k=\bigcap\{\overline{B(\overline{x}_1,\cdots,\overline{x}_k;y)}:(\overline{x}_1,\cdots,\overline{x}_k)\in\prod \limits_{i=1}^{k} \mathcal{S}_i^{N_i}\}.
\end{align*}
Obviously, $\mathcal{F}_k$ is a closed subset of $X$ and $\mathcal{F}_{k+1}\subset \mathcal{F}_{k}$. Define 
\begin{align*}
	\mathcal{F}=\bigcap\limits_{k=1}^{\infty}\mathcal{F}_{k}
\end{align*}
The above construction implies for each $p\in \mathcal{F}$ shadows the points in $\mathcal{S}_i$ for some $i$ with the gap segments $m(\epsilon)$ by the specification property. For any $n>0$ we denotes $n_{rel}$ by the segment of times which shadow the separated points in $\mathcal{S}_i$ for some $i\geq1$. The following lemma shows that $\mathcal{F}\subset E(z_0)$.

\begin{lem}\label{lem3.1}
	For any $x\in \mathcal{F}$, then $x\in E(z_0)$, i.e. $\mathcal{F}\subset E(z_0)$.
\end{lem}
\begin{proof}
	Since $B_{2M+m(\epsilon)}(z_0,\epsilon)$ is open set which contains $z_0$, we only need to show that $O_{f}(x)\cap B_{2M+m(\epsilon)}(z_0,\epsilon)=\varnothing$. Then $O_{f}(x)\subset X\setminus B_{2M+m(\epsilon)}(z_0,\epsilon)$. Furthermore, $\overline{O_{f}(x)}\subset X\setminus B_{2M+m(\epsilon)}(z_0,\epsilon)$, which implies that $z_0\notin \overline{O_{f}(x)}$.
	
	Now we assume that $O_{f}(x)\cap B_{2M+m(\epsilon)}(z_0,\epsilon)\neq\varnothing$ and we can choose $f^j(x)\in B_{2M+m(\epsilon)}(z_0,\epsilon)$. By the construction of $\mathcal{F}$, for any $k$ with $t_k\gg j$, there exists some $\vec{x}_1,\cdots,\vec{x}_k$ such that $x\in B(\vec{x}_1,\cdots,\vec{x}_k;y)$. Hence, we can choose $q<2M+m(\epsilon)$ such that $d(f^{j+q}x,y)<\epsilon$ and $d(f^{j+q}x,f^qz_0)<\epsilon$. Then we have 
	\begin{align*}
		d(y,f^qz_0)\leq d(f^{j+q}x,y)+d(f^{j+q}x,f^qz_0)\leq\epsilon+\epsilon<2\epsilon_0,
	\end{align*}
	which contracts with (\ref{gs3.2}) i.e.,
	$$d(y,\overline{O_f(z_0)})\geq 2\epsilon_0.$$
\end{proof}

\subsubsection{\bf Construction of a special sequence of measures $\mu_k$.}
For each $(\overline{x}_1,\cdots,\overline{x}_k)\in\prod \limits_{i=1}^{k} \mathcal{S}_i^{N_i}$, we choose $z(\overline{x}_1,\cdots,\overline{x}_k;y)\in B(\overline{x}_1,\cdots,\overline{x}_k;y)$. Let $L_k$ be the set of all points constructed in this way. The following simple lemma shows that 
\begin{align}
	\# L_k=\prod \limits_{i=1}^{k} (\#\mathcal{S}_i)^{N_i}.
\end{align}

\begin{lem}\label{lem3.2}
	Let $\overline{x}$ and $\overline{y}$ be distinct elements of $\prod \limits_{i=1}^{k} \mathcal{S}_i^{N_i}$. Then $z_1=z(\overline{x})$ and $z_2=z(\overline{y})$ are $(t_k,5\epsilon)$ separated points.
\end{lem}
\begin{proof}
	Assume that $\overline{x}=(\overline{x}_1,\overline{x}_2,\cdots,\overline{x}_k)$ and $\overline{y}=(\overline{y}_1,\overline{y}_2,\cdots,\overline{y}_k)$ and $\overline{x}_i\neq\overline{y}_i$ with $\overline{x}_s=\overline{y}_s$ for each $s<i,\;1\leq i\leq k$. Let $\overline{x}_i=(x_1^i,\cdots,x_{N_i}^i)$ and $\overline{y}_i=(y_1^i,\cdots,y_{N_i}^i)$. Without lose of generality, we assume that $x_q^i\neq y_q^i$ and for each $u<q$, $x_u^i= y_u^i$.
	
	Then we have for each $0\leq j\leq c_i-1,\;0\leq s\leq M-1$
	\begin{align*}
		d(f^{j(M+2m+1)+s}f^{t_{i-1}+(q-1)(m+\hat{n}_i)+m}z(\overline{x}),f^{jM+s}x_q^i)<\epsilon
	\end{align*}
	and 
	\begin{align*}
		d(f^{j(M+2m+1)+s}f^{t_{i-1}+(q-1)(m+\hat{n}_i)+m}z(\overline{y}),f^{jM+s}y_q^i)<\epsilon
	\end{align*}	
	Since $x_q^i\neq y_q^i\in \mathcal{S}_i$ are $(n_i,7\epsilon)$-separated points, there exists $0\leq \hat{j}\leq c_i-1,\;0\leq \hat{s}\leq M-1$ such that	
	\begin{align*}
		d(f^{\hat{j}M+\hat{s}}x_q^i),f^{\hat{j}M+\hat{s}}y_q^i)\geq 7\epsilon.
	\end{align*}	
	Hence 	
	\begin{align*}
		&d_{t_k}(z_1,z_2)\geq d_{t_i}(z_1,z_2)\\
		\geq& d(f^{\hat{j}(M+2m+1)+\hat{s}}f^{t_{i-1}+(q-1)(m+\hat{n}_i)+m}z(\overline{x}),f^{\hat{j}(M+2m+1)+\hat{s}}f^{t_{i-1}+(q-1)(m+\hat{n}_i)+m}z(\overline{y})\\
		\geq& d(f^{\hat{j}M+\hat{s}}x_q^i,f^{\hat{j}M+\hat{s}}y_q^i)-d(f^{\hat{j}(M+2m+1)+\hat{s}}f^{t_{i-1}+(q-1)(m+\hat{n}_i)+m}z(\overline{x}),f^{\hat{j}M+\hat{s}}x_q^i)\\
		-&d(f^{\hat{j}(M+2m+1)+\hat{s}}f^{t_{i-1}+(q-1)(m+\hat{n}_i)+m}z(\overline{y}),f^{\hat{j}M+\hat{s}}y_q^i)\\
		\geq&7\epsilon-\epsilon-\epsilon=5\epsilon.
	\end{align*}
	So we have done.
\end{proof}

We now define the measures on $\mathcal{F}$ which yield the required estimates for the similar entropy distribution principle. For each $k$, an atomic measure centered on $L_k$. Precisely, if $z=z(\overline{x}_1,\cdots,\overline{x}_k)$, we define probability measure
\begin{align*}
	\mu_k:=\frac{1}{\#L_k}\sum_{z\in L_k}\delta_z.
\end{align*} 

In order to prove the main results of this paper, we present some lemmas.
\begin{lem}
	The sequence of measures $\{\mu_k\}_{k\in \mathbb{N}}$ converges to a measure in $\mathcal{M}(X)$ with respect to
	the weak$^*$-topology $\mu$. Furthermore, the
	limiting measure $\mu$ satisfies $\mu(\mathcal{F})=1$.
\end{lem}

\begin{proof}
	The similar proof as \cite[Lemma 5.4]{TV03} can be applied to show $\mu_k$ converges in the weak$^*$-topology. 
	
	Suppose $\mu$ is a limit measure of the sequence of probability measures $\mu_k$. Then $\mu=\lim\limits_{k}\mu_{s_k}$ for some $s_k\rightarrow\infty$. For some fixed $s$ and all $p\geq0$, we have $\mu_{s+p}(\mathcal{F}_{s})=1$ since $\mu_{s+p}(\mathcal{F}_{s+p})=1$ and $\mathcal{F}_{s+p}\subset\mathcal{F}_{s}$. Therefore, 
	\begin{align*}
		\mu(\mathcal{F}_{s})\geq\limsup\limits_{k\to\infty}\mu_{s_k}(\mathcal{F}_{s})=1.
	\end{align*}
	It follows that $\mu(\mathcal{F})=\lim\limits_{s\rightarrow\infty}\mu(\mathcal{F}_{s})=1$.
\end{proof}

Next we set $b_n$ denote the mistake segment which at most $n$ i.e.,
\begin{align*}
	b_n:=n-n_{rel}.
\end{align*}
Let $\mathcal{B}=B_n(q,\epsilon)$ be an arbitrary ball which intersects $\mathcal{F}$. Let $k$ be the unique number which satisfies
$t_{k}\leq n<t_{k+1}.$ Let $j\in\{0,1,\cdots,N_{k+1}-1\}$ be the unique number so
\begin{align*}
	t_{k}+j(\hat{n}_{k+1}+m(\epsilon))\leq n<t_{k}+(j+1)(\hat{n}_{k+1}+m(\epsilon))
\end{align*}
Let $\Delta_j^{k+1}:=j(\hat{n}_{k+1}+m(\epsilon))$, we have 
\begin{align*}
	t_{k}+\Delta_j^{k+1}\leq n<t_{k}+\Delta_{j+1}^{k+1}
\end{align*}
We assume that $j\geq1$ and the simpler case $j=0$ is similar.

\begin{lem}\label{lem3.4}
	For any $p\geq1$, suppose $\mu_{k+p}(\mathcal{B})>0$, we have 
	\begin{align*}
		\mu_{k+p}(\mathcal{B})\leq \frac{1}{\#L_k\cdot (\#\mathcal{S}_{k+1})^j}
	\end{align*}
	where $b_n$ denote the length of mistake segment.
\end{lem}
\begin{proof}
	\begin{enumerate}
		\item [(1)]Case $p=1$. Suppose $\mu_{k+1}(\mathcal{B})>0$, then $L_{k+1}\cap \mathcal{B}\neq\emptyset$. Let $z=z(\overline{x},\overline{x}_{k+1})\in L_{k+1}\cap \mathcal{B}$, where $\overline{x}=(\overline{x}_1,\cdots,\overline{x}_k)\in  \mathcal{S}_1^{N_1}\times\cdots\times \mathcal{S}_k^{N_k}$ and $\overline{x}_{k+1}=(x_1^{k+1},\cdots,x_{N_i}^{k+1})\in \mathcal{S}_{k+1}^{N_{k+1}}$. Let 
		\begin{align*}
			\mathcal{A}_{\overline{x};x_1^{k+1},\cdots,x_j^{k+1}}=\left\lbrace z(\overline{x},(y_1^{k+1},\cdots,y_{N_{k+1}}^{k+1}))\in L_{k+1}:y_1^{k+1}=x_1^{k+1},\cdots,y_j^{k+1}=x_j^{k+1} \right\rbrace 
		\end{align*}
		Suppose that $z'=z(\overline{y},\overline{y}_{k+1})\in L_{k+1}\cap \mathcal{B}$. Since $d_n(z,z')<2\epsilon$, by Lemma \ref{lem3.2}, we have $\overline{x}=\overline{y}$ and $y_l^{k+1}=x_l^{k+1}$ for $l\in \{1,\cdots,j\}$. Thus  we have 
		\begin{align*}
			\mu_{k+1}(\mathcal{B})&=\frac{1}{\#L_{k+1}}\sum_{z\in L_{k+1}}\delta_z(\mathcal{B})\\
			&\leq \sum\limits_{z\in \mathcal{A}_{\overline{x};x_1^{k+1},\cdots,x_j^{k+1}}}\frac{1}{\#L_{k+1}}\delta_z(\mathcal{B})\\
			&\leq \frac{\#\mathcal{S}_{k+1}^{N_{k+1}-j}}{\#L_{k+1}}=
			\frac{1}{\#L_{k}\cdot(\#\mathcal{S}_{k+1})^{j}}.
		\end{align*}
		\item[(2)]Case $p>1$. Similarly, we have 
		\begin{align*}
			\mu_{k+p}(\mathcal{B})&\leq \frac{\#\mathcal{S}_{k+1}^{N_{k+1}-j}\cdot\#\mathcal{S}_{k+2}^{N_{k+2}}\cdots\#\mathcal{S}_{k+p}^{N_{k+p}}}{\#L_{k+p}}\\
			&=\frac{1}{\#L_{k}\cdot(\#\mathcal{S}_{k+1})^{j}}. 
		\end{align*}
	\end{enumerate}
	
\end{proof}

\begin{lem}\label{lem3.5}
	There exists $N\in \mathbb{N}$ such that for any $n\geq N$,
	\begin{align*}
		\mu(\mathcal{B})\leq \exp\left\lbrace -n(C-2\gamma)|\log7\epsilon| \right\rbrace
	\end{align*}
\end{lem}
\begin{proof}
	By (\ref{gs3.3}), we have 
	\begin{align*}
		\#L_{k}\cdot(\#\mathcal{S}_{k+1})^{j}&= \#\mathcal{S}_{1}^{N_{1}}\cdots\#\mathcal{S}_{k}^{N_{k}}\cdot\#\mathcal{S}_{k+1}^{j}\\
		&\geq\exp\left\lbrace \left( \sum_{i=1}^{k}N_in_i+j \right)(C-\gamma)|\log7\epsilon|\right\rbrace \\
		&=\exp\left\lbrace (n-b_n)(C-\gamma)|\log7\epsilon| \right\rbrace.	 
	\end{align*}
	By Lemma \ref{lem3.4}, we get
	\begin{align*}
		\mu_{k+p}(\mathcal{B})&\leq \frac{1}{\#L_k\cdot (\#\mathcal{S}_{k+1})^j}
		\leq\exp\left\lbrace -n(C-\gamma)|\log7\epsilon| +b_n(C-\gamma)|\log7\epsilon| \right\rbrace\\
		&\leq\exp\left\lbrace -n(C-\gamma)|\log7\epsilon|+b_nC|\log7\epsilon| \right\rbrace\\
		&\leq\exp\left\lbrace -n(C-2\gamma)|\log7\epsilon| \right\rbrace  
	\end{align*}
	The above inequality follows from (\ref{gs3.6}) that 
	$\frac{Cb_n}{n}<\gamma$.
	So 
	\begin{align*}
		\mu(\mathcal{B})&\leq \liminf\limits_{p\to\infty}\mu_{k+p}(\mathcal{B})
		\leq \exp\left\lbrace -n(C-2\gamma)|\log7\epsilon| \right\rbrace.
	\end{align*}
	Hence the desired result follows. 
\end{proof}

\subsubsection{\bf Apply pressure distribution principle type argument.}
Now we are able to finish the proof of Theorem 1.1 by using the pressure distribution principle type argument.

Let $N$ be the number defined in Lemma \ref{lem3.5}. Let $\varGamma=\{B_{n_i}(x_i,\epsilon)\}_{i\in I}$ be any finite cover of $\mathcal{F}$
with $n_i\geq N$ for all $i\in I$. Without loss of generality, we may assume that $B_{n_i}(x_i,\epsilon)\cap\mathcal{F}\neq\emptyset$ for
every $i\in I$. Applying Lemma \ref{lem3.5} on each $B_{n_i}(x_i,\epsilon)$, one has
\begin{align*}
	\sum_{i\in I}\exp\left\lbrace -n_i(C-2\gamma)|\log7\epsilon| \right\rbrace
	\geq\sum_{i\in I}\mu(B_{n_i}(x_i,\epsilon))\geq\mu(\mathcal{F})=1
\end{align*}
As $\varGamma$ is arbitrary, one has
\begin{align*}
	m(\mathcal{F},(C-2\gamma)|\log7\epsilon|,N,\epsilon)\geq1>0
\end{align*}
Therefore, by the fact that $m(\mathcal{F},(C-4\gamma)|\log7\epsilon|,N,\epsilon)$ does not decrease as $N$ increases,
\begin{align*}
	m(\mathcal{F},(C-2\gamma)|\log7\epsilon|,\epsilon)\geq1>0
\end{align*}
which implies that
\begin{align*}
	h_{top}^B(\mathcal{F},f,\epsilon)\geq (C-2\gamma)|\log7\epsilon|.
\end{align*}
So, by Lemma \ref{lem3.1}, (\ref{gs3.4}) and (\ref{gs3.5}), we have 
\begin{align*}
	C-2\gamma&\leq\frac{h_{top}^B(\mathcal{F},f,\epsilon)}{|\log7\epsilon|}\leq \frac{h_{top}^B(E(z_0),f,\epsilon)}{|\log7\epsilon|}=\frac{h_{top}^B(E(z_0),f,\epsilon)}{|\log\epsilon|}\cdot\frac{|\log\epsilon|}{|\log7\epsilon|}\\
	&\leq({\rm\overline{mdim}}^B_M(E(z_0),f,d)+\gamma)\cdot\frac{|\log\epsilon|}{|\log7\epsilon|}\leq{\rm\overline{mdim}}^B_M(E(z_0),f,d)+2\gamma. 
\end{align*}  
Thus, ${\rm\overline{mdim}}^B_M(E(z_0),f,d)\geq C-4\gamma$. As $\gamma >0$ and $ C$ are arbitrary, we obtain $${\rm\overline{mdim}}^B_M(E(z_0),f,d)\geq{\rm\overline{mdim}}_M(X,f,d).$$  

\subsection{Proof for the case of the lower metric mean dimension}
In this subsection, we briefly prove the following equation
\begin{align*}
	{\rm\underline{mdim}}_M^B(E(z_0),f,d)={\rm\underline{mdim}}_M(X,f,d)
\end{align*}
under the assumptions $E(z_0)\neq\emptyset$.

Proposition \ref{prop} and (\ref{gs2.1}) imply ${\rm\underline{mdim}}^B_M(E(z_0),f,d)\leq{\rm\underline{mdim}}_M(X,f,d)$. In the following, we prove
${\rm\underline{mdim}}^B_M(E(z_0),f,d)\geq{\rm\underline{mdim}}_M(X,f,d)$.

We fix any constant $C'<{\rm\underline{mdim}}_M(X,f,d)$.
Next, we only need to show that
\begin{align}\label{gs3.8}
	{\rm\underline{mdim}}^B_M(E(z_0),f,d)\geq C'.
\end{align}  

Fix $\gamma>0$ we can choose a $\epsilon'<\epsilon_0$ and a sequence $\{n_k\}_{k\geq1}\subset\mathbb{N}$ such that there exists maximal $(n_k,7\epsilon')$-separated set $\mathcal{S}'_k$ of $X$ which is always a  $(n_k,7\epsilon')$-spanning set such that 
\begin{align}\label{gs3.9}
	\#\mathcal{S}'_k\geq\exp\left(n_k(C-\gamma)\right)|\log7\epsilon'|,
\end{align}
\begin{align}\label{gs3.10}
	\frac{h^B_{top}(E(z_0),f,\epsilon')}{|\log\epsilon'|}\leq{\rm\underline{mdim}}^B_M(E(z_0),f,d)+\gamma
\end{align}
and
\begin{align}\label{gs3.11}
	({\rm\underline{mdim}}^B_M(E(z_0),f,d)+\gamma)\cdot\frac{|\log\epsilon'|}{|\log7\epsilon'|}\leq{\rm\underline{mdim}}^B_M(E(z_0),f,d)+2\gamma.
\end{align}

We can use the parallel proof in the subsection 3.1.1 and subsection 3.1.2 to show that there exist a
Moran-like fractal $\mathcal{F}'$ and a measure $\mu'$ concentrated on $\mathcal{F}'$ satisfying the following property.

\begin{lem}\label{lem3.6}
	There exists $N'\in \mathbb{N}$ such that for any $n\geq N'$, if $B_n(z,\epsilon')\cap\mathcal{F}'\neq\emptyset$, then
	\begin{align*}
		\mu'(B_n(z,\epsilon'))\leq \exp\left\lbrace -n(C'-2\gamma)|\log7\epsilon'| \right\rbrace.
	\end{align*}
\end{lem}

Let $N'$ be the number defined in Lemma \ref{lem3.6}. Let $\varGamma=\{B_{n_i}(x_i,\epsilon')\}_{i\in I}$ be any finite cover of $\mathcal{F}'$
with $n_i\geq N'$ for all $i\in I$. Without loss of generality, we may assume that $B_{n_i}(x_i,\epsilon')\cap\mathcal{F}'\neq\emptyset$ for
every $i\in I$. Applying Lemma \ref{lem3.6} on each $B_{n_i}(x_i,\epsilon')$, one has
\begin{align*}
	\sum_{i\in I}\exp\left\lbrace -n_i(C'-2\gamma)|\log7\epsilon'| \right\rbrace
	\geq\sum_{i\in I}\mu(B_{n_i}(x_i,\epsilon'))\geq\mu(\mathcal{F}')=1
\end{align*}
As $\varGamma$ is arbitrary, one has
\begin{align*}
	m(\mathcal{F}',(C'-2\gamma)|\log7\epsilon'|,N,\epsilon')\geq1>0
\end{align*}
Therefore, by the fact that $m(\mathcal{F}',(C'-2\gamma)|\log7\epsilon'|,N,\epsilon')$ does not decrease as $N$ increases,
\begin{align*}
	m(\mathcal{F}',(C'-2\gamma)|\log7\epsilon'|,\epsilon')\geq1>0
\end{align*}
which implies that
\begin{align*}
	h_{top}^B(\mathcal{F}',f,\epsilon')\geq (C'-2\gamma)|\log7\epsilon'|.
\end{align*}
So, by Lemma \ref{lem3.1}, (\ref{gs3.10}) and (\ref{gs3.11}), we have 
\begin{align*}
	C'-2\gamma&\leq\frac{h_{top}^B(\mathcal{F}',f,\epsilon')}{|\log7\epsilon'|}\leq \frac{h_{top}^B(E(z_0),f,\epsilon')}{|\log7\epsilon'|}=\frac{h_{top}^B(E(z_0),f,\epsilon')}{|\log\epsilon'|}\cdot\frac{|\log\epsilon'|}{|\log7\epsilon'|}\\
	&\leq({\rm\underline{mdim}}^B_M(E(z_0),f,d)+\gamma)\cdot\frac{|\log\epsilon'|}{|\log7\epsilon'|}\leq{\rm\underline{mdim}}^B_M(E(z_0),f,d)+2\gamma. 
\end{align*}  
Thus, ${\rm\underline{mdim}}^B_M(E(z_0),f,d)\geq C'-4\gamma$. As $\gamma >0$ and $C'$ are arbitrary, we obtain $${\rm\underline{mdim}}^B_M(E(z_0),f,d)\geq{\rm\underline{mdim}}_M(X,f,d).$$  

The proof of Theorem \ref{th1} is complete.

\noindent {\bf Acknowledgements.} The work was supported by the
National Natural Science Foundation of China (Nos.12071222 and 11971236), China Postdoctoral Science Foundation (No.2016M591873), 
and China Postdoctoral Science Special Foundation (No.2017T100384). The work was also funded by the Priority Academic Program Development of Jiangsu Higher Education Institutions.  We would like to express our gratitude to Tianyuan Mathematical Center in Southwest China(No.11826102), Sichuan University and Southwest Jiaotong University for their support and hospitality.

\end{document}